\documentclass[12pt,a4paper,twoside,reqno]{amsart}
\allowdisplaybreaks

\usepackage[all,cmtip]{xy}

\usepackage{amsmath,amscd}
\usepackage{amssymb,amsfonts}
\usepackage[driver=pdftex,margin=3cm,heightrounded=true,centering]{geometry}
\usepackage{mathtools}
\usepackage{tensor}
\usepackage{url}
\usepackage[colorlinks=true,linkcolor=blue,citecolor=blue]{hyperref}
\usepackage{slashed}

\usepackage{graphicx}

\usepackage{thmtools}

\usepackage{amsthm}

\theoremstyle{plain}

\newtheorem{theorem}{Theorem}[section]
\newtheorem*{thm*}{Theorem}
\newtheorem{prop}[theorem]{Proposition}
\newtheorem{lemma}[theorem]{Lemma}
\newtheorem{cor}[theorem]{Corollary}

\theoremstyle{definition}

\theoremstyle{remark} 

\newtheorem{remark}[theorem]{Remark}

\theoremstyle{plain}

\usepackage{amscd}

\numberwithin{equation}{section}

\newcommand{\alpheqn}[1][\relax]{
     \refstepcounter{equation}
     \if#1\relax \relax
       \else \label{#1}
     \fi  
     \setcounter{saveeqn}{\value{equation}}%
    \setcounter{equation}{0}%
    \renewcommand{\theequation}{\thealphequation}}
\newcommand{\reseteqn}{\setcounter{equation}{\value{saveeqn}}%
     \renewcommand{\theequation}{\thearabicequation}}

\providecommand{\mathscr}{\mathcal} 
\IfFileExists{mathrsfs.sty}{
\usepackage{mathrsfs}}         
   {
     \IfFileExists{eucal.sty}{
        \usepackage[mathscr]{eucal}} 
   {
   }
}

\newcommand{\vertiii}[1]{{\left\vert\kern-0.25ex\left\vert\kern-0.25ex\left\vert #1 
    \right\vert\kern-0.25ex\right\vert\kern-0.25ex\right\vert}}
\newcommand{\Bvert}[1]{{\Big\vert\kern-0.25ex\Big\vert\kern-0.25ex\Big\vert #1 
    \Big\vert\kern-0.25ex\Big\vert\kern-0.25ex\Big\vert}}
\newcommand{\bvert}[1]{{\big\vert\kern-0.25ex\big\vert\kern-0.25ex\big\vert #1 
    \big\vert\kern-0.25ex\big\vert\kern-0.25ex\big\vert}}
\newcommand{\nvert}[1]{{\vert\kern-0.25ex\vert\kern-0.25ex\vert #1 
    \vert\kern-0.25ex\vert\kern-0.25ex\vert}}

\renewcommand{\leq}{\leqslant}
\renewcommand{\geq}{\geqslant}

\newcommand{\cd}{\cdot}
\newcommand{\clc}{\cdot\ldots\cdot}
\newcommand{\ot}{\otimes}

\newcommand{\olo}{\otimes\ldots\otimes}

\newcommand{\ci}{\circ}

\newcommand{\ti}{\times}
\newcommand{\nn}{\mathbb{N}}
\newcommand{\zz}{\mathbb{Z}}

\newcommand{\cc}{\mathbb{C}}

\newcommand{\al}{\alpha}

\newcommand{\de}{\delta}
\newcommand{\De}{\Delta}

\newcommand{\io}{\iota}

\newcommand{\la}{\lambda}

\newcommand{\Om}{\Omega}
\newcommand{\si}{\sigma}

\newcommand{\te}{\theta}

\newcommand{\ov}{\overline}
\newcommand{\C}[1]{\mathcal{#1}}

\newcommand{\T}[1]{\textup{#1}}

\newcommand{\B}[1]{\mathbb{#1}}

\newcommand{\fork}[2]{\left\{ \begin{array}{#1} #2 \end{array} \right.}

\newcommand{\su}{\subseteq}

\newcommand{\q}{\qquad}
\newcommand{\qq}{\qquad \qquad}

\newcommand{\wit}{\widetilde}
\newcommand{\wih}{\widehat}

\newcommand{\inn}[1]{\langle #1 \rangle}

\newcommand{\binn}[1]{\big\langle #1 \big\rangle}

\newcommand{\sem}{\setminus}

\makeatletter
\@namedef{subjclassname@2020}{%
  \textup{2020} Mathematics Subject Classification}
\makeatother

\begin{document}



\title[The Haar state on the Vaksman-Soibelman quantum spheres]{The Haar state on the Vaksman-Soibelman quantum spheres}

\author{Max Holst Mikkelsen}
\address{Department of Mathematics and Computer Science,
The University of Southern Denmark,
Campusvej 55, DK-5230 Odense M,
Denmark}
\email{maxmi@imada.sdu.dk}

\author{Jens Kaad}
\address{Department of Mathematics and Computer Science,
The University of Southern Denmark,
Campusvej 55, DK-5230 Odense M,
Denmark}
\email{kaad@imada.sdu.dk}

\begin{abstract}
In this note we present explicit formulae for the Haar state on the Vaksman-Soibelman quantum spheres. Our formulae correct various statements appearing in the literature and our proof is straightforward relying simply on properties of the modular automorphism group for the Haar state.
\end{abstract}

\maketitle
\tableofcontents

\section{Introduction}
One of the main tools in the analysis of compact quantum groups is the Haar state which is the correct quantum analogue of the Haar measure for compact Hausdorff groups. The existence of the Haar state on compact quantum groups was established by Woronowicz in \cite{Wor:CMP,Wor:CQG} (assuming the existence of a faithful state) and in the general case by Van Daele, \cite{VDa:HMC}. In applications, it is of course highly relevant to have explicit formulae for the Haar state instead of merely knowing the existence of such a functional. 

We are in this text interested in the Vaksman-Soibelman quantum spheres which arise as homogeneous spaces for the quantized version of the special unitary group. These higher quantum spheres were introduced and studied by Vaksman and Soibelman in \cite{VaSo:AQS} and these authors also gave an explicit formula for the Haar state (on polynomial expressions in the generators).

Unfortunately, there are some misprints both in the original formula for the Haar state from \cite{VaSo:AQS} and in the updated formula from \cite{SoVa:PTQ}. We therefore found it relevant to finally present a correct version of this formula. Moreover, we give a different and more elementary proof relying only on some basic computations coming from properties of the modular automorphism of the Haar state. This is in line with the original computations due to Woronowicz for the case of quantum $SU(2)$, \cite[Appendix 1]{Wor:CMP}. 

The above mentioned misprints have been corrected by Sheu in the formula appearing in \cite[Theorem 10]{She:CQG}, but in Sheu's work no argument is given for the validity of these corrections.

We end this paper by applying our main theorem to show that the Haar states, for different values of the deformation parameter $q \in (0,1]$, form a continuous field of faithful states on the Vaksman-Soibelman quantum spheres. This can also be derived from the more general work of Nagy in \cite{Nag:HQG,Nag:DQP}, but our proof is, at least in our opinion, much more straightforward.

\subsubsection*{Acknowledgements}
The authors gratefully acknowledge the financial support from  the Independent Research Fund Denmark through grant no.~9040-00107B, 7014-00145B and 1026-00371B.

We would also like to thank David Kyed for many nice discussions on the content of the present paper and to thank Francesco D'Andrea, Piotr Hajac and Tomasz Maszczyk for urging us to present more details in the proof of Proposition \ref{p:diagonal} and Proposition \ref{p:injective}.

\section{Preliminaries}
We start out by reviewing the basic theory relating to the quantized special unitary group and the corresponding Vaksman-Soibelman spheres. We pay particular attention to the arguments leading to an explicit formula for the modular automorphism of the Haar state. Indeed, these arguments are a bit difficult to extract from the present literature on the subject.

\subsection{Quantum $SU(N)$}
We fix $\ell\in \mathbb{N}$ and $q\in (0,1]$. Moreover, we put $N := \ell + 1$.

  Let $\C O( M_q(N))$ denote the universal unital algebra over $\cc$ with $N^2$ generators $u_{ij}$, indexed by $i,j \in \{1,2,\ldots,N\}$, subject to the relations
  \[
  \begin{split}
    & u_{ik} u_{jk} = q u_{jk} u_{ik} \q u_{ki} u_{kj} = q u_{kj} u_{ki} \qq i < j \\
    & u_{il} u_{jk} = u_{jk} u_{il} \qq i < j \, , \, \, k < l \\
    & u_{ik} u_{jl} - u_{jl} u_{ik} = (q - q^{-1}) u_{jk} u_{il} \qq i < j \, , \, \, k < l .
  \end{split}
  \]

  For each $n \in \nn$ we let $S_n$ denote the group of permutations of the set $\{1,2,\ldots,n\}$. For a permutation $\si \in S_n$ we recall that the inversion number $\io(\si)$ denotes the number of inversions in the permutation $\si$.

  For each pair of non-empty subsets $I,J \su \{1,2,\ldots,N\}$ of the same cardinality $|I| = |J| = n$ we may choose
  $1 \leq i_1 < i_2 < \ldots < i_n \leq N$ and $1 \leq j_1 < j_2 < \ldots < j_n \leq N$ such that $I = \{i_1,i_2,\ldots,i_n\}$ and $J = \{j_1,j_2,\ldots,j_n\}$. The \emph{quantum $n$-minor determinant} with respect to the row $I$ and column $J$ is defined by
  \[
D_{IJ} := \sum_{\si \in S_n} (-q)^{\io(\si)} u_{i_{\si(1)} j_1} u_{i_{\si(2)} j_2} \clc u_{i_{\si(n)} j_n} \in \C O(M_q(N)) .
  \]

  In the case where $I = J = \{1,2,\ldots,N\}$ we put $D_{IJ} := D_q$ and refer to $D_q$ as the quantum determinant. Moreover, for $i,j \in \{1,2,\ldots,N\}$ we apply the notation $A_{ij} \in \C O(M_q(N))$ for the quantum $\ell$-minor determinant with respect to the row $\{1,2,\ldots,N\} \sem \{i\}$ and column $\{1,2,\ldots,N\} \sem \{j\}$.

  We let $\C O(SL_q(N))$ denote the quotient of $\C O(M_q(N))$ by the two-sided ideal generated by $D_q - 1 \in \C O(M_q(N))$. The following result can be found in \cite[Section 9.2.3, Proposition 10]{KlSc:QGR}:

  \begin{prop}\label{p:hopf}
    There is a unique Hopf algebra structure on the unital $\cc$-algebra $\C O(SL_q(N))$ with comultiplication $\De\colon  \C O(SL_q(N)) \to \C O(SL_q(N)) \ot \C O(SL_q(N))$ and counit $\epsilon \colon \C O(SL_q(N)) \to \cc$ determined by
    \[
\De(u_{ij}) = \sum_{k = 1}^N u_{ik} \ot u_{kj} \q \mbox{ and } \q \epsilon(u_{ij}) = \de_{ij},
\]
where $\de_{ij} \in \{0,1\}$ denotes the Kronecker delta. The antipode $S \colon \C O(SL_q(N)) \to \C O(SL_q(N))$ is given by the formula
\[
S(u_{ij}) = (-q)^{i-j} A_{ji}. 
\]
Moreover, it holds that $S^2(u_{ij}) = q^{2(i - j)} u_{ij}$.
  \end{prop}

  The Hopf algebra $\C O(SL_q(N))$ becomes a Hopf $*$-algebra when equipped with the involution determined by $u_{ij}^* := S(u_{ji}) =  (-q)^{j-i} A_{ij}$. We denote this Hopf $*$-algebra by $\C O(SU_q(N))$ and refer to it as the \emph{coordinate algebra} of the \emph{quantum group} $SU_q(N)$.

For a Hilbert space $H$, we denote the unital $C^*$-algebra of bounded operators on $H$ by $\B B(H)$. We let $C(SU_q(N))$ denote the universal unital $C^*$-algebra associated with $\C O(SU_q(N))$ and refer to $C(SU_q(N)) $ as \textit{quantum} $SU(N)$. Thus, on $\C O(SU_q(N))$ we first consider the seminorm
  \[
\| x \| := \sup\big\{ \| \pi(x) \| \mid \pi \colon \C O(SU_q(N)) \to \B B(H) \text{ is a unital $*$-homomorphism} \big\} .
\]
Letting $\C I$ denote the kernel of the seminorm $\| \cd \|$ we record that $\C I$ is a $*$-ideal inside $\C O(SU_q(N))$ and that the above formula defines a norm on the quotient space $\C O(SU_q(N))/ \C I$. The corresponding completion agrees with $C(SU_q(N))$. When equipped with the coproduct induced by the coproduct on $\C O(SU_q(N))$, we get that $C(SU_q(N))$ is a compact quantum group in the sense of Woronowicz, see \cite{Wor:CQG} and \cite[Proposition 1.1.4]{NeTu:CQG} for a short proof. 

We shall see later on in this text that the induced unital $*$-homomorphism $j\colon \C O(SU_q(N)) \to C(SU_q(N))$ is in fact injective.


\subsection{The Haar state}
We immediately record the following fundamental result which is a consequence of \cite[Theorem 1.3]{Wor:CQG}:

\begin{theorem}
  There exists a unique state $h\colon C(SU_q(N)) \to \cc$ satisfying that
  \[
(\T{id} \ot h)\De(x) = h(x) \cd 1= (h \ot \T{id}) \De(x)
\]
for all elements $x \in C(SU_q(N))$. 
\end{theorem}

We refer to the state $h\colon C(SU_q(N)) \to \cc$ as the \emph{Haar state} on quantum $SU(N)$. For the remainder of this subsection we shall explain how the Haar state can be applied to prove the injectivity of the unital $*$-homomorphism $j \colon \C O(SU_q(N)) \to C(SU_q(N))$. To this end, it does of course suffice to show that the composition $h \ci j \colon \C O(SU_q(N)) \to \cc$ is faithful (meaning that $h(j(x^* x)) > 0$ for all $x \neq 0$). This latter fact can be established by applying some of the representation theory for the coordinate algebra $\C O(SU_q(N))$. We review some of the main steps here below. To ease the notation we put $G := SU_q(N)$.

  Recall that a \emph{matrix corepresentation} of $\C O(G)$ is given by a matrix $v \in M_n\big(\C O(G)\big)$ for some $n \in \nn$ such that the relations
  \[
\De(v_{ij}) = \sum_{k = 1}^n v_{ik} \ot v_{kj} \q \T{and} \q \epsilon(v_{ij}) = \de_{ij} 
\]
hold for all $i,j \in \{1,2,\ldots,n\}$. Since $\C O(G)$ is a Hopf algebra, it follows that the matrix corepresentation $v$ is automatically an invertible element in the unital algebra $M_n\big(\C O(G)\big)$.

A matrix corepresentation $v \in M_n\big( \C O(G) \big)$ is called \emph{unitary} when $S(v_{ij}) = v_{ji}^*$ for all $i,j \in \{1,2,\ldots,n\}$. Since $\C O(G)$ is a Hopf $*$-algebra, the unitarity of the matrix corepresentation $v$ is equivalent to $v$ being a unitary element in the unital $*$-algebra $M_n\big( \C O(G) \big)$.

Letting $e_1,e_2,\ldots,e_n$ denote the standard basis vectors for $\cc^n$, a matrix corepresentation $v \in M_n\big( \C O(G) \big)$ induces a linear map
\[
\varphi \colon \cc^n \to \cc^n \ot \C O(G) \q \varphi(e_j) := \sum_{k = 1}^n e_k \ot v_{kj} .
\]
We say that $v$ is \emph{irreducible} when it is impossible to find a proper subspace $U \su \cc^n$ such that $\varphi(U) \su U \ot \C O(G)$ (here proper means $U \neq \{0\}$ and $U \neq \cc^n$). We remark in passing that $\varphi$ is a corepresentation of $\C O(G)$ in the sense that
\[
(\varphi \ot \T{id})\varphi = (\T{id} \ot \De)\varphi \q \T{and} \q (\T{id} \ot \epsilon) \varphi = \T{id} .
\]

Two irreducible matrix corepresentations $v$ and $w$ in $M_n\big( \C O(G) \big)$ are said to be \emph{equivalent} when there exists an invertible matrix $T \in M_n(\cc)$ such that $vT = Tw$. We let $\widehat{\C O(G)}$ denote the set of equivalence classes of irreducible unitary matrix corepresentations. The unit $1 \in \C O(G)$ is an irreducible unitary matrix corepresentation and the corresponding equivalence class is denoted by ${\bf 1} \in \widehat{\C O(G)}$. Moreover, letting $u \in M_N( \C O(G))$ denote the matrix with entries $u_{ij}$, $i,j \in \{1,2,\ldots,N\}$, we see that $u$ is also an irreducible unitary matrix corepresentation. Indeed, irreducibility follows since the matrix entries span an $N^2$-dimensional vector subspace of $\C O(G)$. We refer to the matrix $u$ as the \emph{fundamental matrix corepresentation}.

For each equivalence class $\al \in \widehat{\C O(G)}$ we choose a representative $ u^\al \in M_{n_\al}\big(\C O(G)\big)$. The following \emph{Peter-Weyl decomposition} stated here below is a consequence of \cite[Section 11.1.4, Corollary 10]{KlSc:QGR} and \cite[Section 11.1.5, Proposition 12]{KlSc:QGR}. In the present setting, it hinges on the observation that the matrix entries of the fundamental matrix corepresentation $u \in M_N( \C O(G))$ generate $\C O(G)$ as an algebra over $\cc$. 


\begin{theorem}\label{t:petweyl}
The set of matrix entries $\big\{  u^{\al}_{ij} \mid \al \in \widehat{\C O(G)} \, , \, \, i,j \in \{1,2,\ldots,n_\al \} \big\}$ forms a vector space basis for $\C O(G)$.
\end{theorem}

The Haar functional $\eta \colon \C O(G) \to \cc$ is defined by putting
\[
\eta(1) := 1 \q \T{and} \q \eta( u^{\al}_{ij}) := 0 \, \, \T{for all } \al \in \widehat{\C O(G)} \sem \{{\bf 1}\} \, , \, \, i,j \in \{1,2,\ldots,n_\al\}.
\]
By construction, the Haar functional becomes \emph{bi-invariant} in the sense that
\begin{equation}\label{eq:biinvar}
(\T{id} \ot \eta)\De(x) = \eta(x) \cd 1 = (\eta \ot \T{id})\De(x) \q \T{for all } x \in \C O(G) .
\end{equation}
The Haar functional is in fact \emph{uniquely determined} by the condition $\eta(1) = 1$ and \emph{one of} the identities in \eqref{eq:biinvar}. In particular, we obtain that $\eta = h \ci j$. For any matrix corepresentation $v \in M_n\big(\C O(G) \big)$ we may thus conclude that
\[
\eta(v^* v) = h\big( j(v)^* j(v) \big) > 0,
\]
where the various operations are applied entrywise. The above strict inequality relies on the positivity of the Haar state $h \colon C(SU_q(N)) \to \cc$ and entails that $\eta(v^* v)$ is an invertible complex matrix. Using this observation, the argument in \cite[Theorem 5.2]{Wor:CMP} shows that $v$ is equivalent to a unitary matrix corepresentation.

We now quote the following result, which is a special case of \cite[Section 11.3.2, Proposition 29]{KlSc:QGR}. 

\begin{theorem}
  The Haar functional $\eta \colon \C O(SU_q(N)) \to \cc$ satisfies that $\eta(x^* x) > 0$ for all $x \neq 0$. 
\end{theorem} 

Using the above theorem and the identity $\eta = h \ci j$ we may conclude that $j \colon \C O(SU_q(N)) \to C(SU_q(N))$ is injective. An application of Proposition \ref{p:hopf} now shows that $\big(C(SU_q(N)),u\big)$ is a compact matrix pseudogroup in the sense of \cite{Wor:CMP}. The unital $*$-subalgebra of $C(SU_q(N))$ generated by the matrix entries $u_{ij}$ agrees with $\C O(SU_q(N))$.

\subsection{The modular automorphism of the Haar state}
We remind the reader that $q \in (0,1]$ and that $G = SU_q(N)$. In this subsection we review the modular properties of the Haar functional $\eta = h \ci j \colon \C O(G) \to \cc$.

For a matrix corepresentation $v \in M_n\big(\C O(G)\big)$ we recall that the \emph{contragredient} matrix corepresentation $v^c \in M_n\big(\C O(G)\big)$ is defined by $(v^c)_{ij} := S(v_{ji})$ for all $i,j \in \{1,2,\ldots,n\}$. In particular, we have the \emph{double contragredient} matrix corepresentation $v^{cc} \in M_n\big( \C O(G) \big)$ with entries $(v^{cc})_{ij} = S^2(v_{ij})$.

For each $z \in \cc$ we define the unital algebra homomorphism
\[
f_z \colon \C O(G) \to \cc \q f_z(u_{ij}) := \de_{ij} \cd q^{z\cd(2j - N-1)}  \, \, , \, \, \, i,j \in \{1,2,\ldots,N\} .
\]
 These algebra homomorphisms are compatible with the Hopf algebra structure in the sense that
  \[
(f_w \ot f_z)\De = f_{w + z} \, \, , \, \, \, f_0 = \epsilon \, \, \T{ and } \, \, \,  f_z S = f_{-z}
\]
for all $z,w \in \cc$. It moreover holds that $f_z(x^*) = \ov{ f_{-\ov{z}}(x)}$ and that the family $\{f_z\}_{z \in \cc}$ is \emph{analytic} in the sense that the map $z \mapsto f_z(x)$ is entire whenever $x \in \C O(G)$. 

Each of the algebra homomorphisms $f_z \colon \C O(G) \to \cc$ induces a pair of algebra automorphisms
\[
\si_L^{-iz} \T{ and } \si_R^{-iz} \colon \C O(G) \to \C O(G) 
\]
defined by the formulae $\si_L^{-iz}(x) := (\T{id} \ot f_z)\De(x)$ and $\si_R^{-iz}(x) := (f_z \ot \T{id})\De(x)$ for all $x \in \C O(G)$. The relationship between these algebra automorphisms and the square of the antipode can then be expressed as follows:
  \[
S^2(x) = (\si_R^{-i} \si_L^{i})(x) \q \T{for all } x \in \C O(G) .
\]
For any matrix corepresentation $v \in M_n\big( \C O(G) \big)$, we therefore see that the double contragredient $v^{cc}$ is equivalent with $v$. Indeed, we get that
\[
S^2(v) = (\si_R^{-i} \si_L^{i})(v) = f_1(v) \cd v \cd f_{-1}(v), 
\]
where the different operations are applied entry by entry. This in turn implies that $f_1(v) \cd v = v^{cc} \cd f_1(v)$.

The modular properties of the Haar functional $\eta\colon \C O(G) \to \cc$ can now be stated and verified. Define the algebra automorphism
\[
\te := \si_R^{-i} \si_L^{-i} \colon \C O(G) \to \C O(G) .
\]
In the statement here below we let $\T{Tr} \colon M_n(\cc) \to \cc$ denote the operator trace with $\T{Tr}(1_n) = n$.

\begin{prop}\label{p:modular}
  The algebra automorphism $\te \colon \C O\big(SU_q(N)\big) \to \C O\big( SU_q(N)\big)$ is the modular automorphism for the Haar functional in the sense that
  \[
\eta(x \cd y) = \eta\big(y \cd \te(x) \big) \q \mbox{for all } x,y \in \C O\big(SU_q(N)\big).
\]
Moreover, for every irreducible unitary matrix corepresentation $v \in M_n\big(\C O(G)\big)$ we have that $f_1(v)$ is positive and invertible and that $\T{Tr}\big( f_1(v) \big) = \T{Tr}\big( f_{-1}(v) \big)$.
\end{prop}
\begin{proof}
 For every irreducible unitary matrix corepresentation $v \in M_n\big( \C O(G) \big)$ we know from \cite[Section 11.3.2, Lemma 30]{KlSc:QGR} that there exists a unique invertible and positive operator $F_v \in M_n(\cc)$ satisfying that $F_v \cd v = v^{cc} \cd F_v$ and $\T{Tr}(F_v) = \T{Tr}(F_v^{-1})$.

  From the proof of \cite[Theorem 5.6]{Wor:CMP} we get that there exists a unique algebra homomorphism $g \colon \C O(G) \to \cc$ such that $g(v) = F_v$ whenever $v$ is an irreducible unitary matrix corepresentation. Moreover, this algebra homomorphism $g \colon \C O(G) \to \cc$ induces the modular automorphism $\nu \colon \C O(G) \to \C O(G)$ in the sense that $\nu(x) = (g \ot \T{id} \ot g)(\De \ot \T{id})\De(x)$ for all $x \in \C O(G)$. In order to prove the present proposition we therefore only need to establish that $g = f_1$. 

To see that the identity $g = f_1$ holds, we recall that the matrix entries of the fundamental corepresentation unitary $u \in M_N\big(\C O(G)\big)$ generate $\C O(G)$ as an algebra over $\cc$ and that $u$ is an irreducible unitary matrix corepresentation. In order to finish our proof it therefore suffices to show that $f_1(u) = g(u)$. But this is clear since $f_1(u) \cd u = u^{cc} \cd f_1(u)$ and since $f_1(u)$ is a diagonal $N \ti N$ matrix with entries $q^{1 - N}, q^{3 - N}, \ldots, q^{N-1}$. Indeed, we get that 
\[
\T{Tr}\big( f_1(u)\big) = \sum_{j = 1}^N q^{N - 2j + 1} = \T{Tr}\big( f_{-1}(u) \big). \qedhere
\]
\end{proof}

The next lemma is an immediate consequence of the definition of the modular automorphism $\te \colon \C O(G) \to \C O(G)$.

\begin{lemma}\label{l:modaut}
For every $i,j\in \{1,2,\dots,N \}$ it holds that $\te(u_{ij}) = q^{2( i +j -N- 1)} u_{ij}$. 
\end{lemma}

\subsection{The maximal torus and invariance properties of the Haar state}\label{ss:maxtorus}
Recall that $\ell \in \nn$ and that $N = \ell + 1$. We are in this section going to review how the maximal torus $\B T^\ell$ sits as a compact matrix sub-pseudogroup of $G = SU_q(N)$. This entails some important invariance properties for the Haar state $h \colon C(SU_q(N)) \to \cc$.

We let $C(\B T^\ell)$ denote the unital $C^*$-algebra of continuous functions on the $\ell$-torus. The coordinate projections, which we think of as unitary elements in $C(\B T^{\ell})$, are denoted by $w_1,w_2,\ldots,w_\ell \colon \B T^\ell \to \cc$. We apply the notation
\[
I\colon C(\B T^\ell) \to \cc \q I( w_1^{n_1} w_2^{n_2} \clc w_\ell^{n_\ell}) := \de_{0 n_1} \cd \de_{0 n_2} \clc \de_{0 n_\ell} 
\]
for the faithful state which integrates a continuous function on $\mathbb{T}^\ell $ with respect to the Haar measure on the $\ell$-torus. We let $\C O(\B T^\ell)$ denote the unital $*$-subalgebra of $C(\B T^\ell)$ generated by the coordinate projections and record that $\C O(\B T^\ell)$ becomes a Hopf $*$-algebra with coproduct, antipode and counit given by $\De(w_j) = w_j \ot w_j$, $S(w_j) = w_j^{-1}$ and $\epsilon(w_j) = 1$ for all $j \in \{1,2,\ldots,\ell\}$.

We define the surjective unital $*$-homomorphism
\[
\Phi \colon  C(SU_q(N)) \to C(\B T^\ell) \q \Phi( u_{ij} ) := \fork{ccc}{ \de_{ij} w_i & i < N \\ \de_{ij} w_1^{-1} w_2^{-1} \clc w_\ell^{-1} & i = N }
\]
and observe that $\Phi$ induces a surjective homomorphism $\Phi \colon \C O(SU_q(N)) \to \C O(\B T^\ell)$ of the underlying Hopf $*$-algebras. We may apply $\Phi$ to define the homogeneous spaces
\[
\begin{split}
C(G/\B T^\ell ) & = \big\{ x \in C(G) \mid (\T{id} \ot \Phi)\De(x) = x \ot 1 \big\} \q \T{and} \\
C(\B T^\ell \backslash G) & = \big\{ x \in C(G) \mid (\Phi \ot \T{id})\De(x) = 1 \ot x \big\} .
\end{split}
\]
These homogeneous spaces are unital $C^*$-subalgebras of $C( SU_q(N) )$.

Let us also introduce the linear maps $L$ and $R\colon C(SU_q(N)) \to C(SU_q(N))$ given by the following expressions:
\[
L := (\T{id} \ot I \Phi)\De  \q \T{and} \q R := (I \Phi \ot \T{id})\De .
\]
The next result is a consequence of the discussion in \cite[Section 2]{Nag:HQG}.

\begin{prop} \label{p:LR}
  The linear maps $L$ and $R$ are faithful conditional expectations onto the homogeneous spaces $C(G/\B T^\ell)$ and $C(\B T^\ell \backslash G)$, respectively. Moreover, it holds that $h L = h = h R$ and that $LR = RL$. 
\end{prop}

\subsection{The Vaksman-Soibelman quantum spheres}
We are now going to recall the definition of the Vaksman-Soibelman quantum sphere $C(S_q^{2\ell + 1})$. Our exposition pays particular attention to a commutative $C^*$-subalgebra $D_q^\ell$ which agrees with a quantized version of the standard $\ell$-simplex. The invariance properties of the Haar state under the left and right actions of the maximal torus $\B T^\ell$ described in the previous section, entail that we only need to compute the Haar state on the commutative $C^*$-subalgebra $D_q^\ell$.

For each $j \in \{1,2,\ldots,N\}$ we put $z_j := u_{Nj} \in \C O(SU_q(N))$ and let $\mathcal{O}(S_q^{2\ell+1})$ denote the unital $*$-subalgebra of $\C O(SU_q(N))$ generated by the elements $z_j$. It follows from the defining relations for $\C O(SU_q(N))$ that the generators $z_1,z_2,\ldots,z_N$ satisfy the relations
\begin{equation}\label{eq:quasph}
\begin{split}
z_i z_j &=q z_j z_i  \, \, , \, \, \, i < j \q 
z_i^* z_j = q z_j z_i^* \, \, , \, \, \,  i \neq j\\
[z_{j+1}^*, z_{j+1}] & = (1-q^2)\sum_{i=1}^{j} z_iz_i^* \, \, , \, \, \, j\in \{1,2,\dots,\ell\}   \\
[z_1^*, z_1]& = 0 \q \sum_{i=1}^N z_i z_i^* =1 .
\end{split}
\end{equation}
We refer to $\mathcal{O}(S_q^{2\ell+1})$ as the \textit{coordinate algebra} of the quantum $(2\ell+1)$-sphere and the norm-closure of this coordinate algebra inside $C(SU_q(N))$ shall be denoted by $C(S_q^{2\ell+1})$. We refer to the unital $C^*$-algebra $C(S_q^{2\ell + 1})$ as the \emph{(Vaksman-Soibelman) quantum $(2\ell + 1)$-sphere}. Finally, we apply the notation $D_q^\ell$ for the unital $C^*$-subalgebra of $C(S_q^{2\ell + 1})$ generated by $\{z_j z_j^*\}_{j = 1}^\ell$ and refer to $D_q^\ell$ as the \emph{quantized $\ell$-simplex.} We immediately notice that $D_q^\ell$ is commutative.

Before proving the next lemma, it is convenient to record that $ \mathcal{O}(S_q^{2\ell+1})$ agrees with the linear span of elements of the form $  z_1^{n_1} z_2^{n_2} \clc z_N^{n_N} (z_1^*)^{m_1} (z_2^*)^{m_2} \clc (z_N^*)^{m_N}$ where $n_i,m_i\in \mathbb{N}_0$ for all $i\in \{1,2,\dots,N\} $. 

\begin{lemma} \label{l:imE}
  The composition $LR\colon  C(SU_q(N)) \to C(SU_q(N))$ restricts to a faithful conditional expectation $E\colon  C(S_q^{2\ell + 1}) \to D_q^\ell$. 
\end{lemma}
\begin{proof}
Let $x = z_1^{n_1} z_2^{n_2} \clc z_N^{n_N} (z_1^*)^{m_1} (z_2^*)^{m_2} \clc (z_N^*)^{m_N}\in \mathcal{O}(S_q^{2\ell+1})$. It can be verified that
  \[
  \begin{split}
  R(x) & = \de_{\sum_{j = 1}^N n_j, \sum_{j = 1}^N m_j} \cd x \, \, \T{ and } \\
  L(x) & = \de_{n_1 - m_1,n_N - m_N} \cd \de_{n_2 - m_2, n_N - m_N} \clc \de_{n_\ell - m_\ell, n_N - m_N} \cd x .
  \end{split}
\]
We thereby obtain the identity $LR(x) = \de_{n_1,m_1} \cd \de_{n_2,m_2} \clc \de_{n_N,m_N} \cd x\in D_q^\ell $.
\end{proof}
%

For a countable non-empty index set $I$, we let $\ell^2(I)$ denote the Hilbert space of $\ell^2$-sequences indexed by $I$. The standard basis for $\ell^2(I)$ is denoted by $\{ e_i \}_{i \in I}$. 

In the case where $q \in (0,1)$, we let $W$ and $D\colon \ell^2(\nn_0) \to \ell^2(\nn_0)$ denote the weighted shift operator and the diagonal operator given by
\[
W(e_n):= \sqrt{1 - q^{2(n+1)}} e_{n + 1} \, \, \T{ and } \, \, \, D(e_n) := q^n e_n .
\]
We also consider the unitary operator $U\colon \ell^2(\zz) \to \ell^2(\zz)$ defined by $U(e_m) := e_{m + 1}$.


 \begin{prop}\label{p:repsphere}
Suppose that $q \in (0,1)$. There exists a unital $*$-homomorphism $\pi_N\colon C\big(SU_q(N)\big) \to \B B\big(\ell^2(\nn_0)^{\ot \ell} \ot \ell^2(\zz)\big)$ satisfying that
  \[
  \pi_N(z_j) = \fork{ccc}{ D^{\ot (N - j)} \ot W \ot 1^{\ot (j -2)} \ot 1 & \T{for} & j > 1 \\
    D^{\ot \ell} \ot U & \T{for} & j = 1 } .
  \]
 \end{prop}
\begin{proof}
  For every $n \in \{1,2,\ldots,\ell\}$ we define the unital $*$-homomorphism $\Phi_n \colon  C(SU_q(N)) \to C(SU_q(2))$ by
  \[
  \Phi_n(u_{ij}) := \fork{cc}{ u_{i - n + 1, j - n+1} & i,j \in \{n,n+1\} \\
  \de_{ij} &   \T{otherwise}}  .
  \]
  Notice in this respect that $u_{22}^* = u_{11}$ and $u_{21}^* = -q^{-1} u_{12}$ inside $C(SU_q(2))$. Moreover, we have the unital $*$-homomorphisms
  \[
\rho \colon C(SU_q(2)) \to \B B( \ell^2(\nn_0)) \, \, \T{ and } \, \, \, \pi_2 \colon C(SU_q(2)) \to \B B\big(\ell^2(\nn_0) \ot \ell^2(\zz) \big)
\]
defined by the formulae
\[
\begin{split}
& \rho(u_{21}) := D \, \, \T{and} \, \, \, \rho(u_{22}) := W \\
  & \pi_2(u_{21}) := D \ot U \, \, \T{and} \, \, \, \pi_2(u_{22}) := W \ot 1 .
\end{split}
\]
Letting $\De^{(\ell - 1)}\colon C(SU_q(N)) \to C(SU_q(N))^{\ot \ell}$ denote the $(\ell - 1)$-times iterated coproduct we define the unital $*$-homomorphism
\[
\pi_N := \big( (\rho \Phi_\ell) \ot (\rho \Phi_{\ell - 1}) \olo (\pi_2 \Phi_1) \big) \De^{(\ell - 1)} \colon  C(SU_q(N)) \to \B B\big( \ell^2(\nn_0)^{\ot \ell} \ot \ell^2(\zz) \big) .
\]
A straightforward verification now establishes the claimed identities for the operators $\pi_N(z_j) = \pi_N(u_{Nj})$, $j \in \{1,2,\ldots,N\}$.
\end{proof} 

For $q \in (0,1)$ we introduce the following subset of $[0,1]^\ell$: 
\[
U_{\ell,q} := \big\{ \big(q^{2\sum_{j=1}^\ell \! n_j }, q^{2\sum_{j=1}^{\ell-1} n_j}(1- q^{2n_\ell }),\dots,q^{2n_1}(1-q^{2n_2})\big)  \mid n_1,n_2,\dots ,n_\ell\in \mathbb{N}_0 \big\} . 
\]
We let $\Omega_{\ell,q}$ denote the closure of $U_{\ell,q}$ inside $[0,1]^\ell$. Remark that $U_{\ell,q}$ is isomorphic as a set with $\nn_0^\ell$ via the map
\begin{equation}\label{eq:ennuuu}
  \la(n_1,n_2,\ldots,n_\ell) := \big( q^{2\sum_{j=1}^\ell \! n_j }, q^{2\sum_{j=1}^{\ell-1} n_j}(1- q^{2n_\ell }),\dots,q^{2n_1}(1-q^{2n_2}) \big) .
\end{equation}

For $q = 1$ we let $\Omega_{\ell,q}$ denote the standard $\ell$-simplex inside $[0,1]^\ell$ defined by
\[
\Om_{\ell,1} := \big\{ (t_1,t_2,\ldots,t_\ell) \in [0,1]^\ell \mid \sum_{j = 1}^\ell t_j \leq 1 \big\} . 
\]

For each $j\in \{1,2,\dots ,N\}$ we put $A_j := \sum_{i=1}^j z_iz_i^*$ and remark that $A_N = 1$. We also define $A_0 := 0$. A useful observation is that
\begin{align}
\label{zjzj}
z_j z_j^* =A_j-A_{j-1} \ \  \ \text{and} \ \  \ z_j^* z_j=A_j-q^2 A_{j-1} \q \T{for all } j\in \{1,2,\dots, N\} .
\end{align}

\begin{lemma}\label{l:character}
  Suppose that $q \in (0,1)$. Let $\chi \colon D_q^\ell \to \cc$ be a character and let $j \in \{1,2,\ldots,\ell\}$. If $\chi(A_{j+1}) = q^{2k}$ for some $k \in \nn_0 \cup \{\infty\}$, then $\chi(A_j) = q^{2m}$ for some $m \in \nn_0 \cup \{ \infty\}$ with $m \geq k$.
\end{lemma}
\begin{proof}
  Let $C \su C(S_q^{2\ell + 1})$ denote the smallest unital $C^*$-subalgebra containing $z_1,z_2,\ldots,z_{j+1}$ and let $F \su D_q^\ell$ denote the smallest unital $C^*$-subalgebra containing $z_1 z_1^*, z_2 z_2^*,\ldots, z_{j+1} z_{j+1}^*$. We record that $A_{j+1} = \sum_{i = 1}^{j+1} z_i z_i^*$ belongs to the $C^*$-subalgebra $F \su C$ and that $A_{j+1}$ is a central element in $C$ (since $z_i$ commutes with $z_m z_m^*$ whenever $i < m$).

  Suppose now that $\chi(A_{j+1}) = q^{2k}$ for some $k \in \nn_0 \cup \{\infty\}$. Define the ideals $I_k \su C$ and $J_k \su F$ as the smallest ideals containing $A_{j + 1} - q^{2k}$. The conditional expectation $E \colon C(S_q^{2\ell + 1}) \to D_q^\ell$ from Lemma \ref{l:imE} restricts to a faithful conditional expectation $E \colon C \to F$, and since $A_{j+1}$ is central in $C$ we get that $E$ induces a unital positive map $[ E ] \colon C/I_k \to F/J_k$ between the quotient $C^*$-algebras. In particular, we obtain that the inclusion $F \to C$ induces an injective $*$-homomorphism $F/J_k \to C/I_k$.

Let $\pi \colon C \to C/I_k$ and $p \colon F \to F/J_k$ denote the quotient maps. Since $F/J_k$ can be regarded as a unital $C^*$-subalgebra of $C/I_k$ we get the following identity of spectra:
\begin{equation}\label{eq:idespec}
\T{Sp}\big( p(A_j) \big) = \T{Sp}\big( \pi(A_j) \big) .
\end{equation}

We are now interested in the spectrum of $\pi(A_j) \in C/I_k$. Suppose therefore that $\la$ belongs to the spectrum of $\pi(A_j)$ and that $\la \neq q^{2k}$. Inside $C/I_k$ we have the relations
\begin{equation}\label{eq:relations}
\pi(z_{j+1}) \pi(z_{j+1}^*) = q^{2k} - \pi(A_j) \q \T{and} \q \pi(z_{j+1}^*) \pi(z_{j+1}) = q^{2k} - q^2 \pi(A_j) .
\end{equation}
 Using the following standard identity of spectra
\[
\T{Sp}\big( \pi(z_{j+1}) \pi(z_{j+1}^*) \big) \cup \{0\} = \T{Sp}\big( \pi(z_{j+1}^*) \pi(z_{j+1}) \big) \cup \{0\}
\]
together with \eqref{eq:relations} we deduce that $q^{2k} - \la$ belongs to the spectrum of $q^{2k} - q^2 \pi(A_j)$. This shows that $q^{-2} \la$ lies in the spectrum of $\pi(A_j)$. We conclude from these observations that
\begin{equation}\label{eq:specinc}
\T{Sp}\big( \pi(A_j) \big) \su \big\{ q^{2m} \mid m \in \nn_0 \cup \{ \infty\} \, , \, \, m \geq k \big\} .
\end{equation}

To end the proof of the lemma, we notice that the character $\chi \colon D_q^\ell \to \cc$ induces a character $[\chi] \colon F/J_k \to \cc$. The identity of spectra from \eqref{eq:idespec} together with the inclusion from \eqref{eq:specinc} now entail that
\[
\chi(A_j) = [ \chi]( p(A_j)) \in \T{Sp}\big( p(A_j)\big) \su  \big\{ q^{2m} \mid m \in \nn_0 \cup \{ \infty\} \, , \, \, m \geq k \big\} . \qedhere
\]
\end{proof}

As a consequence of Lemma \ref{l:character} we get that
\[
\chi(A_j) \in \big\{ q^{2k} \mid k \in \nn_0 \big\} \cup \{0\} \q \T{for all } j \in \{1,2,\ldots,\ell\} 
\]
whenever $q \neq 1$ and $\chi \colon D_q^\ell \to \cc$ is a character.

The following result can be found as \cite[Lemma 4.2]{HoSz:QSPS} for $q \in (0,1)$ (but we include the case where $q = 1$ here). Notice however that, comparing with \cite[Lemma 4.2]{HoSz:QSPS}, our Lemma \ref{l:character} provides a more detailed argument for the possible values of characters on the quantized simplex $D_q^\ell$. 

\begin{prop}\label{p:diagonal}
The unital $C^*$-subalgebra $D_q^\ell$ is commutative and $*$-isomorphic to $C(\Omega_{\ell,q})$ via the map which sends $z_jz_j^*$ to the $j$'th coordinate function for every $j\in \{1,2,\ldots,\ell\}$. 
\end{prop}
\begin{proof}
Suppose that $q \neq 1$. For each $j \in \{1,2,\ldots,\ell\}$ we record that 
    \[
    \begin{split}
  & \pi_N(z_j z_j^*)( e_{n_1} \olo e_{n_{\ell}} \ot e_m) \\
  & \q = \fork{ccc}{
q^{2 \sum_{j = 1}^\ell n_j} \cd e_{n_1} \olo e_{n_{\ell}} \ot e_m & \T{for} & j = 1 \\
q^{2 \sum_{j = 1}^{N - j} n_j} \cd (1 - q^{2 n_{N-j+1}}) \cd e_{n_1} \olo e_{n_{\ell}} \ot e_m & \T{for} & j \in \{2,\ldots,\ell\} } 
      \end{split}
\]
whenever the indices $n_1,\ldots,n_{\ell}$ lie in $\nn_0$ and $m$ belongs to $\zz$. This entails that the $*$-homomorphism from $D_q^\ell$ to $C(\Om_{\ell,q})$ sending $z_j z_j^*$ to the $j$'th coordinate function is well-defined and surjective. The injectivity of this $*$-homomorphism is a consequence of Lemma \ref{l:character} and Gelfand duality.

Suppose now that $q = 1$. In this case, we may identify $C(S_q^{2\ell + 1})$ with the unital $C^*$-algebra of continuous functions on the $(2 \ell + 1)$-dimensional unit sphere $S^{2 \ell + 1} \su \cc^N$. The relevant isomorphism is realized by mapping $z_j$ to the $j^{\T{th}}$ coordinate projection $p_j\colon S^{2\ell + 1} \to \cc$. The unit sphere $S^{2 \ell + 1}$ comes equipped with an action of the $N$-torus $\B T^N \su \cc^N$ given by
  \[
  (\xi_1,\xi_2,\ldots,\xi_N) \cd (\la_1,\la_2,\ldots,\la_N) := ( \xi_1 \cd \la_1, \xi_2 \cd \la_2,\ldots,\xi_N \cd \la_N)
  \]
 for all $(\la_1,\la_2,\ldots,\la_N) \in \B T^N$ and $(\xi_1,\xi_2,\ldots,\xi_N) \in S^{2\ell + 1}$. The corresponding quotient space is isomorphic to the standard $\ell$-simplex $\Om_{\ell,1} \su [0,1]^\ell$. At the level of unital $C^*$-algebras we get that $D_1^\ell \cong C(S^{2\ell + 1}/\B T^N)$ and this establishes our claim.
\end{proof} 

For $q \neq 1$, the proof of Proposition \ref{p:diagonal} also shows that the restriction of the representation $\pi_N$ to the quantized simplex $D_q^\ell$ yields an injective $*$-homomorphism $\pi_N \colon D_q^\ell \to \B B\big( \ell^2(\nn_0)^{\ot \ell} \ot \ell^2(\zz) \big)$. This result is applied in the proof of the following proposition: 


\begin{prop}\label{p:injective}
Suppose that $q \in (0,1)$. The unital $*$-homomorphism $\pi_N \colon C(S_q^{2\ell + 1}) \to \B B\big( \ell^2(\nn_0)^{\ot \ell} \ot \ell^2(\zz) \big)$ is injective.
\end{prop}
\begin{proof}
  For each $\la = (\la_1,\la_2,\ldots,\la_N) \in \B T^N$ we define the unitary operator
  \[
  \begin{split}
  & U_\la \colon \ell^2(\nn_0)^{\ot \ell} \ot \ell^2(\zz) \to \ell^2(\nn_0)^{\ot \ell} \ot \ell^2(\zz) \\
  & U_\la( e_{n_1} \olo e_{n_\ell} \ot e_m) := \la_1^m \la_2^{n_\ell} \clc \la_N^{n_1} \cd e_{n_1} \olo e_{n_\ell} \ot e_m .
    \end{split}
  \]
  In this fashion we obtain a strongly continuous unitary representation of the $N$-torus on the Hilbert space $\ell^2(\nn_0^{\ot \ell}) \ot \ell^2(\zz)$. For each $\la \in \B T^N$ and $j \in \{1,2,\ldots,N\}$ we record that $U_\la \pi_N(z_j) U_\la^{-1} = \pi_N(\la_j \cd z_j )$ and we therefore obtain a faithful conditional expectation $\wit{E} \colon \pi_N( C(S_q^{2\ell + 1})) \to \pi_N( C(S_q^{2\ell + 1}))$ by putting
  \[
\inn{\eta, \wit{E}(x)\xi} := \int_{\B T^N} \binn{\eta, U_\la x U_\la^{-1} \xi} \, d \la
\]
for every $x \in \pi_N(C(S_q^{2\ell + 1}))$ and $\xi,\eta \in \ell^2(\nn_0^{\ot \ell}) \ot \ell^2(\zz)$. Comparing with the formula, provided in the proof of Lemma \ref{l:imE}, for the faithful conditional expectation $E \colon C(S_q^{2\ell + 1}) \to C(S_q^{2\ell + 1})$ we see that
\[
\wit{E}( \pi_N(x)) = \pi_N( E(x)) \q \T{for all } x \in C(S_q^{2\ell + 1}).
\]
The result of the present proposition now follows from the faithfulness of $\wit{E}$ and $E$ together with the injectivity of the restriction of $\pi_N$ to the quantized $\ell$-simplex $D_q^\ell$. 
\end{proof}

\begin{remark}
  Let us for a little while consider the universal unital $C^*$-algebra $C^*(z_1,z_2,\ldots,z_N)$ with generators $z_1,z_2,\ldots,z_N$ subject to the relations in \eqref{eq:quasph}. The argumentation carried out in the present section also shows that the representation $\pi_N$ becomes injective when defined on $C^*(z_1,z_2,\ldots,z_N)$ instead of $C(S_q^{2\ell + 1})$. Notice in this respect that for $C^*(z_1,z_2,\ldots,z_N)$ we obtain the conditional expectation $E$ appearing in Lemma \ref{l:imE} from the strongly continuous action of the $N$-torus determined by $(\la_1,\la_2,\ldots,\la_N) \cd z_j := \la_j \cd z_j$ for all $(\la_1,\la_2,\ldots,\la_N) \in \B T^N$ and $j \in \{1,2,\ldots,N\}$. We therefore get that $C(S_q^{2\ell + 1})$ becomes $*$-isomorphic to the universal $C^*$-algebra $C^*(z_1,z_2,\ldots,z_N)$.
\end{remark}

%

\section{A formula for the Haar state on the quantum spheres}
       We are now ready to present our computation of the Haar state on the Vaksman-Soibelman quantum sphere $C(S_q^{2\ell + 1})$. As indicated earlier, we may restrict our attention to the smaller commutative $C^*$-subalgebra $D_q^\ell$. Our methods only work well for $q \in (0,1)$ since we rely on the non-triviality of the modular automorphism for the Haar state. It is nonetheless important to keep in mind the classical case where $q = 1$ and we therefore include it in the statement of our main theorem here below.


For each $m = (m_1,m_2,\ldots,m_\ell) \in \nn_0^\ell$ we put 
\[
|m| := \sum_{i = 1}^\ell m_i \q \T{and} \q A^m := A_1^{m_1} \cd A_2^{m_2} \clc A_\ell^{m_\ell} .
\]
For each $j \in \{1,2,\ldots,\ell\}$ we shall identify $\nn_0^j$ with a subset of $\nn_0^\ell$ via the injective map
\[
\io \colon \nn_0^j \to \nn_0^\ell \q \io(m_1,m_2,\ldots,m_j) := (m_1,m_2,\ldots,m_j, 0, \ldots, 0) .
\]

\begin{lemma}\label{l:h(xAj)}
 Suppose that $q \neq 1$. Let $j\in \{1,2,\dots,\ell\}$ and $m \in \io\big( \mathbb{N}_0^j \big)$. It holds that 
\begin{align*}
h( A^m A_j)=\frac{1-q^{2(j+|m| )}}{1-q^{2(N+|m|)}} h( A^m) .
\end{align*}
\end{lemma}
\begin{proof}
  We specify that $m = (m_1,m_2,\ldots,m_j,0,\ldots,0)$ for some $m_1,m_2,\ldots,m_j \in \nn_0$.

  Let $k\in \{ j,j+1,\dots,\ell\}$. We have from Lemma \ref{l:modaut} that $\theta(z_{k+1}^*) = q^{-2k} z_{k+1}^* $. Thus, using Proposition \ref{p:modular} and \eqref{eq:quasph}, we compute that 
\[
\begin{split}
  q^{2(k+|m| )}h( A^m z_{k+1}^* z_{k+1} )
  = q^{2(k+|m|)}h(z_{k+1}^* z_{k+1} A^m) 
  & = q^{2|m|} h(z_{k+1} A^m z_{k+1}^* ) \\
  & =h( A^m z_{k+1} z_{k+1}^* ) .
\end{split}
\]
Hence, by applying (\ref{zjzj}), it follows that 
\[
q^{2(k+|m| )}\cd \big( h(A^m A_{k+1})-q^2 (A^m A_{k}) \big) =
h(A^m A_{k+1})-h(A^mA_k)
\]
or equivalently 
\begin{align*}
h(A^mA_k) = \frac{1-q^{2(k+|m|)}}{1-q^{2(k+1+|m|)}} h(A^m A_{k+1} ) .
\end{align*}
Hence, by proceeding inductively on $ k $, we deduce that 
\begin{align*}
h(A^m A_j)=\frac{1-q^{2( j+|m|)}}{1-q^{2(k+1+|m| )}} h(A^m A_{k+1} ) 
\end{align*}
for every $k\in \{ j,j+1,\dots ,\ell\}$. Since $ A_N = A_{\ell+1}=1 $, the desired formula follows. 
\end{proof}

\begin{lemma}\label{l:monomials}
 Suppose that $q \neq 1$. For every $m = (m_1,m_2,\ldots,m_\ell) \in \nn_0^\ell$ it holds that
  \[
h(A^m) = \prod_{k = 1}^\ell \frac{1 - q^{2k}}{1 - q^{2( k + \sum_{i = 1}^k m_i)}} .
  \]
\end{lemma}
\begin{proof}
  To ease the notation, put $C := \prod_{k = 1}^\ell (1 - q^{2k})$. The proof runs by induction on $m \in \nn_0^\ell$. The induction start confirms that $h(1) = 1$. For the induction step we let $j \in \{1,2,\ldots,\ell\}$ and suppose that 
\[
h(A^m) = C \cd \prod_{k = 1}^\ell \big(1 - q^{2( k + \sum_{i = 1}^k m_i)}\big)^{-1}
\]
for some $m \in \io( \nn_0^j)$. An application of Lemma \ref{l:h(xAj)} then shows that
\[
\begin{split}
h(A^m  A_j) & = \frac{1 - q^{2(j + |m|)}}{1 - q^{2(N + |m|)}} \cd h(A^m) \\
& = C \cd \frac{1 - q^{2(j + |m|)}}{1 - q^{2(N + |m|)}} \cd \prod_{k = 1}^\ell \big( 1 - q^{2( k + \sum_{i = 1}^k m_i)}\big)^{-1} \\
& = C \cd \prod_{k = 1}^{j-1} \big( 1 - q^{2( k + \sum_{i = 1}^k m_i)}\big)^{-1} \cd \prod_{k = j}^\ell \big( 1 - q^{2( k + 1 + |m|)}\big)^{-1} .
\end{split}
\]
This confirms the validity of the induction step.
\end{proof}

For $q \neq 1$ we introduce a probability measure $\mu$ on the power set $\C P( \Om_{\ell,q})$ (which in this case agrees with the Borel $\si$-algebra). This probability measure is defined by $\mu(X) := 0$ for $X \cap U_{\ell,q} = \emptyset$ and by the formula
\[
\mu\big( \{ \la(n) \} \big) := \prod_{k = 1}^\ell (1 - q^{2k}) q^{2 (N -k) n_k} \q \T{for all } n \in \nn_0^\ell .
\]
Notice that we are here identifying $U_{\ell,q}$ with $\nn_0^\ell$ via the isomorphism of sets  from \eqref{eq:ennuuu}, to wit
\[
\la(n_1,n_2,\ldots,n_\ell) = \big( q^{2 \sum_{i = 1}^\ell n_i}, q^{2\sum_{i=1}^{\ell-1} n_i}(1- q^{2n_\ell }),\dots,q^{2n_1}(1-q^{2n_2}) \big) .
\]

For $q = 1$ we define the probability measure $\mu := \ell! \cd \nu$ where $\nu$ denotes the Lebesgue measure restricted to the Borel $\si$-algebra for the standard $\ell$-simplex $\Om_{\ell,1} \su [0,1]^\ell$.

\begin{theorem}\label{t:formula}
  For every $x\in C(S_q^{2\ell+1})$ it holds that
\begin{align*}
h(x)=   \int E(x) \, d\mu .
\end{align*}
In particular, if $x = z_1^{n_1} z_2^{n_2} \clc z_N^{n_N} (z_1^*)^{m_1} (z_2^*)^{m_2} \clc (z_N^*)^{m_N}$ where all the exponents belong to $\nn_0$, we have
\begin{align*}
h(x)=    \int \de_{n_1,m_1} \cd \de_{n_2,m_2} \clc \de_{n_N,m_N}\cdot x \, d\mu .
\end{align*}
\end{theorem} 
\begin{proof}
We focus on the case where $q \neq 1$. Recall that Lemma \ref{l:imE} shows that $E\colon C(S_q^{2\ell+1}) \to C(\Omega_{\ell,q})$ is a conditional expectation. Moreover, it is a consequence of Proposition \ref{p:LR} that $h = hE$. Hence, it is sufficient to consider $x\in C(\Omega_{\ell,q})$. Now, since the set $\{ A^m \mid m \in \mathbb{N}_0^\ell  \}$ is linearly dense in $C(\Omega_{\ell,q})$, we may restrict to the case where $x = A^m$ for some $m = (m_1,m_2,\ldots,m_\ell) \in \nn_0^\ell$. In this case, we have that
\[
x\big( \la(n) \big)
= q^{2 m_1 \sum_{i = 1}^\ell n_i} \cd q^{2 m_2 \sum_{i=1}^{\ell-1} n_i} \clc q^{2 m_\ell n_1}
= \prod_{k = 1}^\ell q^{ 2n_k \sum_{i = 1}^{N - k} m_i}
\]
for all $n = (n_1,n_2,\ldots,n_\ell) \in \nn_0^\ell$. It therefore follows that
\begin{equation}\label{eq:integral}
\begin{split}
  \int x \, d\mu
  & = \sum_{n_1,n_2,\ldots,n_\ell = 0}^\infty
\prod_{k = 1}^\ell q^{ 2n_k \sum_{i = 1}^{N - k} m_i} \cd \mu\big( \{\la(n) \} \big) \\
& = \prod_{k = 1}^\ell (1 - q^{2k}) \sum_{n_k = 0}^\infty q^{ 2 n_k (N-k + \sum_{i = 1}^{N - k} m_i)}
= \prod_{k = 1}^\ell \frac{1 - q^{2k} }{1 - q^{2(N-k + \sum_{i = 1}^{N-k} m_i)} } .
\end{split}
\end{equation}
An application of Lemma \ref{l:monomials} shows that the right hand side of \eqref{eq:integral} agrees with $h(A^m)$, and the theorem is therefore proved.
\end{proof}


\section{Applications of the main theorem}
We end this paper by deriving three important corollaries of our main theorem. These three corollaries are also consequences of the work of Nagy who deals with the more general case of quantum $SU(N)$, see \cite{Nag:HQG,Nag:DQP}. Our route to the corollaries here below is however quite different from the route followed by Nagy.

\begin{cor}
The Haar state $h \colon C(S_q^{2\ell + 1}) \to \cc$ is faithful. 
\end{cor}
\begin{proof}
This follows immediately from Theorem \ref{t:formula} since both the conditional expectation $E \colon C(S_q^{2 \ell + 1}) \to C(\Om_{\ell,q})$ and the state on $C(\Om_{\ell,q})$ associated with the probability measure $\mu$ are faithful.
\end{proof}

For the second and the third corollary we fix a $\de \in (0,1)$ and consider the universal unital $C^*$-algebra $C(S_\bullet^{2\ell + 1})$ with $N + 1$ generators $z_1^\bullet, z_2^\bullet, \ldots, z_N^\bullet$ and $\xi$ subject to the relations
\begin{equation}\label{eq:quasphfie}
  \begin{split}
 z^\bullet_i z^\bullet_j &= \xi z^\bullet_j z^\bullet_i  \, \, , \, \, \, i < j \q 
(z^\bullet_i)^* z^\bullet_j = \xi z^\bullet_j (z^\bullet_i)^* \, \, , \, \, \,  i \neq j\\
\big[(z^\bullet_{j+1})^*, z^\bullet_{j+1}\big] & = (1-\xi^2)\sum_{i=1}^{j} z^\bullet_i(z^\bullet_i)^* \, \, , \, \, \, j\in \{1,2,\dots,\ell\}   \\
\big[(z^\bullet_1)^*, z^\bullet_1\big]& = 0 \q \sum_{i=1}^N z^\bullet_i (z^\bullet_i)^* =1 
\end{split}
\end{equation}
together with the requirement that $\xi$ be selfadjoint and central with spectrum contained in the closed unit interval $[\de,1]$. For each $q \in [\de,1]$ we then have the evaluation $*$-homomorphism $\T{ev}_q \colon C(S_\bullet^{2\ell + 1}) \to C(S_q^{2\ell + 1})$ given by $\T{ev}_q(z_j^{\bullet}) = z_j^q$ and $\T{ev}_q(\xi) = q$.

 \begin{cor}
  The Haar states for different values of $q \in [\de,1]$ form a continuous field of states in the sense that the map
  \[
[\de,1] \to \cc \q q \mapsto (h \ci \T{ev}_q)(x^\bullet)
\]
is continuous for every $x^{\bullet} \in C(S_\bullet^{2\ell + 1})$.
\end{cor}
\begin{proof}
  By density and linearity, it suffices to prove the result when $x^{\bullet}$ is a monomial in the operators $z_1^{\bullet}, (z_1^{\bullet})^*, \ldots, z_N^{\bullet}, (z_N^{\bullet})^*$. However, using that $h E = h$ we may restrict our attention to showing that the map
  \[
[\de,1] \to \cc \q q \mapsto h\big((A^q)^m \big)
\]
is continuous whenever $m = (m_1,m_2,\ldots,m_\ell) \in \nn_0^\ell$. We specify here that the superscript $q$ indicates that the usual monomial $A^m$ is considered as an operator in $C(\Om_{\ell,q})$.

The continuity of the expression $h\big((A^q)^m\big)$ on the interval $[\de,1)$ follows immediately from Lemma \ref{l:monomials} and we therefore only need to establish that
  \begin{equation}\label{eq:haarlim}
    \lim_{q \to 1} h\big((A^q)^m\big) = h\big((A^1)^m\big) .
    \end{equation}
Another application of Lemma \ref{l:monomials} shows that the left hand side of \eqref{eq:haarlim} agrees with
\[
\lim_{q\to 1} h\big( (A^q)^m \big) = \ell! \cd \prod_{k = 1}^\ell (k+\sum_{i=1}^k m_i)^{-1}.
\]
Regarding the right hand side of \eqref{eq:haarlim}, an exercise in calculus shows that
\[
\begin{split}
h\big((A^1)^m\big) & = \ell! \cd \int_{\Om_{\ell,1}} t_1^{m_1} (t_1 + t_2)^{m_2} \clc \big( \sum_{j = 1}^\ell t_j \big)^{m_\ell} d\nu \\
& = \ell! \cd \int_0^1 \int_0^{1 - t_1} \ldots \int_0^{1 - \sum_{j = 1}^{\ell - 1} t_j} t_1^{m_1} (t_1 + t_2)^{m_2} \clc (\sum_{j = 1}^\ell t_j)^{m_\ell}
\, dt_\ell \ldots dt_1 \\
& = \ell! \cd \prod_{k = 1}^\ell (k+\sum_{i=1}^k m_i)^{-1}.
\end{split}
\]
This proves the present corollary.
\end{proof} 

We may equip $C(S_\bullet^{2\ell + 1})$ with the structure of a unital $C([\de,1])$-$C^*$-algebra by sending $C([\de,1])$ to the center of $C(S_\bullet^{2\ell + 1})$ via the unital $*$-homomorphism which maps the identity function to $\xi$. Applying the evaluation $*$-homomorphism $\T{ev}_q$ we may identify the fibres for different values of $q \in [\de,1]$ with the corresponding Vaksman-Soibelman quantum sphere $C(S_q^{2\ell + 1})$. A combination of the above two corollaries with \cite[Th\'eor\`eme 3.3]{Bla:DCH} then yields the following:

\begin{cor}
The unital $C([\de,1])$-$C^*$-algebra $C(S_\bullet^{2\ell + 1})$ is a continuous field of unital $C^*$-algebras with fibres given by the Vaksman-Soibelman quantum spheres $C(S_q^{2\ell + 1})$ for all $q \in [\de,1]$. 
\end{cor}

\bibliographystyle{amsalpha-lmp}

\providecommand{\bysame}{\leavevmode\hbox to3em{\hrulefill}\thinspace}
\providecommand{\MR}{\relax\ifhmode\unskip\space\fi MR }
\providecommand{\MRhref}[2]{%
  \href{http://www.ams.org/mathscinet-getitem?mr=#1}{#2}
}
\providecommand{\href}[2]{#2}

\end{document}